\documentclass[oneside,english]{amsart}

\usepackage[latin1]{inputenc} 

\usepackage{amssymb} \usepackage[]{epsf,epsfig,amsmath,amssymb,amsfonts,latexsym}

\newtheorem{theorem}{Theorem}[section]
\newtheorem{lemma}[theorem]{Lemma}

\newtheorem{cor}[theorem]{Corollary}

\theoremstyle{definition}
\newtheorem{defn}[theorem]{Definition}
\newtheorem{remark}[theorem]{Remark}

\newcommand{\PP}{\mathbb{P}}
\newcommand{\EE}{\mathbb{E}}
\newcommand{\stab}{\mathit{stab}}
\newcommand{\Map}{\mathit{Map}}
\newcommand{\OO}{\mathcal{O}}
\newcommand{\OOO}{\OO[M,\delta,\epsilon,F_1,F_2]}

\newcommand{\cov}{\mathit{cov}}
\newcommand{\sep}{\mathit{sep}}
\newcommand{\Prob}{\mathit{Prob}}
\newcommand{\SubG}{\mathit{Sub}_G}
\newcommand{\Av}{\frac{1}{|V|}\sum_{v \in V}}

\title{Positive sofic entropy implies finite stabilizer}
\author{Tom Meyerovitch}
\begin{document}
\maketitle
\begin{abstract}
We prove that for a measure preserving action of a sofic group with positive sofic entropy, the stabilizer is finite on a set of  positive measure.
This extends results of Weiss  and Seward for amenable groups and free groups, respectively.
It follows that  the action of a sofic group on its subgroups by inner automorphisms has zero topological sofic entropy, and that a faithful action that has completely positive sofic entropy must be free.
\end{abstract}

\section{Introduction}

The last decade brought a number of important developments in dynamics of  non-amenable group actions. Among these we note the various extensions of classical entropy theory. For actions of free groups, L. Bowen introduced a numerical  invariant known as \emph{$f$-invariant entropy} \cite{MR2630067}.
%Both amenable groups and free groups belong to the class of \emph{sofic groups}.
Some time later Bowen defined new invariants for actions of
\emph{sofic  groups},  called \emph{sofic entropy} \cite{MR2552252}. Kerr and Li  further developed  sofic entropy theory and also adapted it to groups actions on topological spaces by homeomorphisms \cite{MR2854085}.
% *** Should say here that ``sofic entropy'' is actually a family of invariants that a-priori depend on the approximation sequence? ** .
The classical mean-entropy for amenable groups and Bowen's $f$-invariant both turned out to be special cases of sofic entropy \cite{MR2653969,MR3068400}.
 %was introduced by Bowen \cite{MR2552252} and further developed  by Kerr and  Li   \cite{MR3095712,MR3068400}, and others.

The study of \emph{non-free measure-preserving  group actions} is another fruitful and active trend in dynamics. These are closely related to  the notion of \emph{invariant random subgroups}: Namely, a probability measure on the space of subgroups whose law is invariant under conjugation. Any such law can be realized as the law of the stabilizer for a random point for some probability  preserving action \cite{MR3165420}.
In this note we  prove the following:
\begin{theorem}\label{thm:positive_entropy_almost_free}
Suppose $G \curvearrowright (X,\mathcal{B},\mu)$ is an action of a countable sofic group $G$ that has  positive sofic entropy (with respect to some sofic approximation).
Then the  set of points  in $X$ with finite stabilizer has positive measure.
In particular, if the action is ergodic,  almost every point has finite stabilizer.
\end{theorem}

The amenable case, Weiss  pointed out that the conclusion of Theorem \ref{thm:positive_entropy_almost_free} as a remark in the last section of his survey paper on actions of amenable groups \cite{MR2052281}. To be precise, Weiss stated the amenable case of Corollary \ref{cor:completely_positive_entropy} below.
%the following consequence of Theorem \ref{thm:positive_entropy_almost_free}:

Another interesting case of Theorem \ref{thm:positive_entropy_almost_free} for free groups is due to Seward \cite{1205.5090}. The result proved in   \cite{1205.5090} applies to the \emph{random sofic approximation}. By a non-trivial result of Bowen this coincides with the \emph{$f$-invariant}  for  free groups. Sewards's proof  in \cite{1205.5090} is based on a specific  formula for $f$-entropy, which does not seem to be available for sofic entropy in general. Our proof below proceeds essentially by proving a combinatorial statement about finite objects. In personal communication, Seward informed me of another proof of Theorem \ref{thm:positive_entropy_almost_free} that is  expected to appear in a forthcoming paper of  Alpeev and Seward  as a byproduct of their study of an entropy theory for general countable groups.

%One should keep in mind  that some intuitive features of  entropy theory no longer hold beyond the realm of amenable groups. For instance, sofic entropy is not monotone under taking factors.
Theorem \ref{thm:positive_entropy_almost_free} confirms the point of view that
the ``usual''  notions of sofic entropy for sofic groups (or mean-entropy in the amenable case) are not very useful as invariants for non-free actions.  A  version of sofic entropy for actions with stabilizers  was developed by Bowen \cite{MR3286052}   as a particular instance of a more general framework of entropy theory of sofic groupoids. It seems likely that both the statement of Theorem \ref{thm:positive_entropy_almost_free} and our proof should have a  generalization to sofic class bijective extensions of groupoids. We will not pursue this direction.

\textbf{Acknowledgments.} I thank Yair Glasner, Guy Salomon, Brandon Seward and  Benjy Weiss for  interesting discussions.

\section{Notation and definitions}

\subsection{Sofic groups}

Sofic groups were introduced by Gromov \cite{MR1694588} (under a different name) and by Weiss \cite{MR1803462} towards the end of the millennium.
Sofic groups retain some properties of finite groups. They are
a common generalization of amenable and residually finite groups. We include a definition below. There are  several  other interesting equivalent definitions. There are many good references in  the literature for further background, motivation and discussions on sofic groups, for instance \cite{MR2460675}.

Throughout we will use the notation $F \Subset G$ to indicate that $F$ is a finite subset of $G$.
For a finite set $V$, let $S_V$ denote the group of permutations over $V$.
We will consider maps from a group $G$ to $S_V$. These maps are not necessarily  homomorphisms. Given  a map $\xi:G \to S_V$,  $g \in G$ and $v \in V$, we write $\xi_g \in S_V$ for the image of $g$ under $\xi$ and $\xi_g(v) \in V$  for the image of $v$ under the permutation $\xi_g$.

Let $F \Subset G$ and $\epsilon>0$.
A map $\xi:G \to S_V$  is called an \textbf{ $(F,\epsilon)$-approximation of $G$} if it satisfies the following properties:
\begin{equation}\label{eq:almost_hom}
\frac{1}{|V|} \# \{ v \in V~:~ \xi_g(\xi_h(v)) \ne \xi_{gh}(v) \} < \epsilon ~ \forall g,h \in F.
\end{equation}
and
\begin{equation}\label{eq:almost_free}
\frac{1}{|V|}| \# \{ v \in V~:~ \xi_g(v) = v \} < \epsilon ~ \forall g \in F\setminus \{1\}.
\end{equation}

A \textbf{sofic group} is a group $G$ that admits an $(F,\epsilon)$-approximation for any $F \Subset G$ and any $\epsilon >0$.

A  \textbf{symmetric  $(F,\epsilon)$-approximation of $G$} is $\xi:G \to S_V$ that in addition to  \eqref{eq:almost_hom} and \eqref{eq:almost_free} also satisfies
\begin{equation}\label{eq:almost_symmetric}
\xi_g(\xi_{g^{-1}}(v)) = v  ~ \forall g \in F\setminus \{1\}~,~ v \in V.
\end{equation}

Standard arguments show that a sofic group admits a symmetric $(F,\epsilon)$-approximation for any $F \Subset G$ and any $\epsilon >0$, so from now assume our $(F,\epsilon)$-approximations
 also satisfy \eqref{eq:almost_symmetric}.

A sequence $(\xi_n)_{n=1}^\infty$  of maps $\xi_n:G \to V_n$ is called a \textbf{sofic approximation for $G$} if $$\{ n \in \mathbb{N}~:~ \xi_n \mbox{ is an } (F,\epsilon)-\mbox{approximation}\}$$
is co-finite in $\mathbb{N}$, for any $F \Subset G$ and any $\epsilon >0$.

\subsection{Sofic entropy}

Roughly speaking the sofic entropy of an action is $h$ if  there are ``approximately'' $e^{h|V|}$ ``sufficiently distinct good approximations'' for the action that ``factor through'' a finite ``approximate action''  $\xi:G \to S_V$. Various  definitions have been introduced in the literature, that have been shown to lead to an equivalent notion. Most definitions involve  some auxiliary structure.
 Here we  follow a recent presentation of sofic entropy by Austin \cite{austin2015additivity}.  Ultimately, this presentation is  equivalent to Bowen's original definition and also to  definitions given by Kerr and Li.

Let $G \curvearrowright (X,\mu)$ be a probability preserving action on a standard probability space. As explained  in \cite{austin2015additivity}, by passing to an isomorphic action we can assume without loss of generality that $X = \chi^G$, where $\chi$ is a compact metric space and that the action of $G$ is the shift action: $g(x)_h = x_{g^{-1}h}$, and that $\mu \in \Prob(\chi^G)$ is a Borel probability measure on $\chi^G$, where the Borel structure is with respect to the product topology.

More specifically, we will assume that $\chi = \{0,1\}^\mathbb{N}$ is equipped with the metric $d(\omega,\omega') := 2^{-m(\omega,\omega')}$, where $m(\omega,\omega') := \min\{ n \in \mathbb{N}~:~ \omega_n \ne \omega'_n\}$.
%with the product of discrete topology, and that the metric on $\chi$ takes a countable set of values with only $0$ as an accumulation point.
These assumptions can be made without loss of generality. Indeed,  start with an arbitrary (standard) Borel space $X$,  choose a countable sequence $(A_n)_{n=1}^\infty$ of Borel subsets $A_n \subset X$ so that  the smallest $G$-invariant  $\sigma$-algebra containing $\{A_n\}_{n=1}^\infty$ is the Borel $\sigma$-algebra.  There is a $G$-equivariant Borel embedding  of  $x \in X$ to $\hat{x} \in \chi^G$ defined by
$$(\hat{x}_g)_n := 1_{g^{-1}A_n}(x) ~,~ n \in \mathbb{N}~,~ g\in G.$$
Let $\hat \mu \in \Prob((\chi^\mathbb{N})^G)$ denote the  push-forward  measure  of $\mu$, it follows that the $G$-action on $((\chi^\mathbb{N})^G,\hat \mu)$ is isomorphic to the $G$-action on $(X,\mu)$.
Also note that 
\begin{equation}\label{eq:d_range}
\forall \omega,\omega' \in \chi ~ d(\omega,\omega') \in \{0\} \cup \{2^{-n}~:~ n \in \mathbb{Z}_+\}.
\end{equation}
In particular, the diameter of $(\chi,d)$ is $1$.

We recall some definitions and notation that Austin introduced in  \cite{austin2015additivity}:
\begin{defn}
Given $x \in X = \chi^V$, $\xi:G \to S_V$ and $v \in V$,
the \textbf{pullback name of $x$ at $v$}, denoted by $\Pi_v^\xi(x) \in X= \chi^G$ is defined to be:
\begin{equation}
(\Pi_v^\xi(x))_{g^{-1}} := x_{\xi_g(v)}.
\end{equation}
The \textbf{empirical distribution} of $x$ with respect to $\xi$  is defined by:
\begin{equation}
P_x^\xi := \frac{1}{|V|}\sum_{v \in V}\delta_{\Pi_v^\xi(x)}.
\end{equation}
Given a $w^*$-neighborhood  $\OO \subset \Prob(\chi^G)$ of $\mu \in \Prob(\chi^G)$,
the set of \textbf{$(\OO,\xi)$-approximations} for the action  $G \curvearrowright (X,\mu)$
is given by
$$\Map(\OO,\xi) := \{ x \in \chi^V~:~ P_x^\xi \in \OO\}.$$

\end{defn}

In \cite{austin2015additivity}  elements of $\Map(\OO,\xi)$ are called  ``good models''.

The space $\Map(\OO,\xi) \subset \chi^V$, if it is non-empty, is considered as a metric space with respect to the following metric
$$d^V (x,x') := \frac{1}{V} \sum_{ v \in V} d(x_v,x'_v).$$

Given a compact metric space $(Y,\rho)$ and $\delta >0$ we denote by
$\sep_\delta(Y,\rho)$ the  maximal cardinality of a $\delta$-separated set in $(Y,\rho)$, and by
$\cov_\delta(Y,\rho)$ the minimal number of $\rho$-balls of radius $\delta$ needed to cover $Y$.
Let us recall  a couple of classical relations between these quantities.
Because distinct $2\delta$-separated points cannot be in the same $\delta$-ball the following holds:
$$ \sep_{2\delta}(Y,\rho) \le \cov_\delta(Y,\rho).$$
Consider a maximal $\delta$-separated set $Y_0 \subset Y$.
The collection of $\delta$-balls with centers in $Y_0$  covers $Y$. Thus:
$$ \cov_{\delta}(Y,\rho) \le \sep_{\delta}(Y,\rho).$$

\begin{defn}
Let $\Sigma = (\xi_n)_{n=1}^\infty$ be a sofic approximation of $G$, with $\xi_n:G \to S_{V_n}$. The \textbf{$\Sigma$-entropy} (or \textbf{sofic entropy} with respect to $\Sigma$) of $G \curvearrowright (X,\mu)$  is defined by:

\begin{equation}\label{eq:sofic_entropy_def}
 h_\Sigma(\mu) := \sup_{\delta >0} \inf_{\OO \ni \mu} \limsup_{ n \to \infty} \frac{1}{|V_n|}\log \sep_\delta\left( \Map(\OO,\xi_n),d^{V_n}\right),
 \end{equation}

where the infimum is over weak-$*$ neighborhoods $\OO$ of $\mu$ in $\Prob(X)$.%, $\xi: G \to S_{V_\xi}$.
If $\Map(\OO,\xi_n) =\emptyset$ for all large $n$'s, define $ h_\Sigma(\mu) := -\infty$.

\end{defn}

%$$\limsup_{\xi \in \Sigma}f(\xi):= \inf_{ \Sigma' \Subset \Sigma} \sup_{ \xi \in \Sigma \setminus \Sigma'} f(\xi).$$

The key fact is that the quantity $h_\Sigma(\mu)$ does not depend  on the topological model $X=\chi^G$ or on the choice of metric $d$, and is thus an  invariant for the action $G \curvearrowright (X,\mu)$, with respect to isomorphism in the class of probability preserving actions.

\begin{remark}
We recall a slight generalization of $\Sigma$-entropy:  A \textbf{random sofic approximation} is $\Sigma =(P_n)_{n=1}^\infty$ where $P_n \in \Prob((S_{V_n})^G)$  so that
 the conditions \eqref{eq:almost_hom} and \eqref{eq:almost_free} hold ``on average'' with respect to $P_n$  for any $\epsilon >0$ and $F\Subset G$, if $n$ is large enough.

In this case $\Sigma$-entropy is defined by
\begin{equation}\label{eq:random_sofic_entropy_def}
 h_\Sigma(\mu) := \sup_{\delta >0} \inf_{\OO \ni \mu} \limsup_{ n \to \infty} \frac{1}{|V_n|}\log\left(\int \sep_\delta\left( \Map(\OO,\xi),d^{V_n}\right)dP_n(\xi) \right).
 \end{equation}

For the special case $G$ is a free group on $d$ generators and $P_n$ is chosen uniformly among the homomorphisms from $G$ to the group of permutations of $\{1,\ldots,n\}$,
  Bowen proved that $\Sigma$-entropy coincides with the so called $f$-invariant \cite{MR2653969}.
   %that was invented by Bowen prior to sofic-entropy  \cite{MR2653969}.

Our proof of Theorem \ref{thm:positive_entropy_almost_free} applies directly with no changes to random sofic approximations, in particular to $f$-entropy.

\end{remark}

\subsection{Stabilizers and the space of subgroups}

Let $\SubG \subset 2^G$ denote the space of subgroups of $G$. The space $\SubG$ comes with a compact topology, inherited from the product topology on $2^G$. The group $G$ acts on $\SubG$ by inner automorphisms.
Now let $G \curvearrowright X$ be an action of $G$ on a standard Borel space $X$. For $x \in X$ let
\begin{equation}
\stab(x) := \left\{ g \in G~:~ g(x)=x\right\}.
\end{equation}
The map $\stab:X \to \SubG$ is Borel and  $G$-equivariant.

The following fact about  the map $\stab:X \to \SubG$  appears implicitly for instance in \cite{MR2052281}:
\begin{lemma}\label{prop:finite_stab}
Let $G \curvearrowright (X,\mu)$ be an ergodic action of a countable group.
If the action has finite stabilizers, the map
$\stab:X \to \SubG$ induces a finite factor $G \curvearrowright (\SubG, \mu \circ \stab^{-1})$.
\end{lemma}
\begin{proof}
 Suppose $\stab(x)$ is finite on a set of positive measure.  By ergodicity $|\stab(x)| <\infty$ on a set of full measure. Since there are only countably  many finite subgroups, the measure $\mu \circ \stab^{-1} \in \Prob(\SubG)$ must be purely atomic.  To finish the proof note that  a purely atomic invariant  probability measure must be supported on a single finite orbit, if it is ergodic.
% This is because an atom with infinite orbit is a wandering set, violating Poincar\'{e} recurrence.
\end{proof}

Here is a quick corollary of
Theorem  \ref{thm:positive_entropy_almost_free} that concerns the action $G \curvearrowright \mathit{Sub}_G$:
\begin{cor}\label{cor:IRS_zero_entropy}
%Let $G$ be a sofic group.
Let $G$ be an infinite sofic group and $\Sigma$ a sofic approximation sequence. The topological $\Sigma$-entropy of the action $G \curvearrowright \mathit{Sub}_G$ by conjugation is zero.
\end{cor}
\begin{proof}
The variational principle for  $\Sigma$-entropy states that the topological $\Sigma$-entropy of an action $G \curvearrowright X$ is equal to the supremum of the measure-theoretic $\Sigma$-entropy over all $G$-invariant measures \cite{MR2854085}.
It thus suffices to prove that any $G$-invariant measure on $\SubG$ has zero $\Sigma$-entropy.
By Theorem \ref{thm:positive_entropy_almost_free}, it is enough to show that the set $A =  \{ H \in \SubG~:~ |\stab(H)| < \infty\}$  is null.
%We will show that for any such $\mu$, almost any point has infinite stabilizer, and conclude using Theorem \ref{thm:positive_entropy_almost_free}.
%For this, it is enough to consider ergodic measures $\mu$ **** Check! $\Sigma$-entropy is not affine! Is it concave? ***.
Indeed, for any $H \in \SubG$, $H \subset \stab(H)$, because any subgroup is contained in its normalizer.  Thus groups $H \in \SubG$ with finite stabilizer must be finite, so $A$ is a countable set. Suppose $\mu(A) >0$. It follows that $\mu \mid A$ is purely atomic.
As in  Lemma \ref{prop:finite_stab}, each ergodic component of $\mu \mid A$ must be  supported on finite set. An action of an infinite group on a finite set can not have finite stabilizers.
This shows that $\mu(A) =0$.
\end{proof}

\section{Sampling from finite graphs}

In this section we prove an auxiliary result on finite labeled graphs.

We begin with some terminology:
\begin{defn}
A \textbf{finite , simple and directed graph} is a pair $\mathcal{G}= (V,E)$  where $V$ is a finite set and  $E \subset V^2$ (we allow self-loops but no parallel edges).
\begin{itemize}
\item The \textbf{out-degree} and \textbf{in-degree} of $v \in V$ are given by
$$\deg_{\mathit{out}}(v) := \left|\left\{ w \in V~:~ (v,w) \in E\right\}\right|,$$
$$\deg_{\mathit{in}}(v) := \left|\left\{ w \in V~:~ (w,v) \in E\right\}\right|,$$
\item $\mathcal{G}$ is \textbf{$(\epsilon,k,M)$-regular} if  at most $\epsilon |V|$ vertices have out-degree  less than $k$, and all vertices have in-degree at most $M$.
\item A set $W \subset V$ is \textbf{$\epsilon$-dominating} if the number of  vertices in $v \in V$ so that $\{ w \in W~:~ (v,w) \in E\} = \emptyset$ is at most  $\epsilon|V|$.
\item A \textbf{$p$-Bernoulli} set $W \subset V$ for $p \in (0,1)$ is a random subset of $V$ such that for each $v \in V$ the probability that $v \in W$ is $p$, independently of the other vertices.
\end{itemize}
\end{defn}

\begin{lemma}%[Small $\kappa$-dominating random subsets in high degree graphs ]
\label{lem:small_random_set_dominating}
Fix any $\kappa \in (0,1)$. % and $M \ge 1$.
Suppose $k \le M \le  N$ satisfy
\begin{equation}\label{eq:k_M}
(1-\frac{1}{\sqrt{k}})^k< \kappa \mbox{ and } N > 2M^2\kappa^{-3}.
\end{equation}
%there exists $k_0$ sufficiently large such that for all $k > k_0$  and any $M >0$
% there exists $N = N(k,M)$ with the following property:
 For any  $(\kappa,k,M)$-regular  graph $G=(V,E)$ with   $|V| >  N$,
 %with  maximal degree at most $M$,
%there exists a $3\kappa$-dominating set of size at most $\frac{2}{\sqrt{k}} |V|$.
a $\frac{1}{\sqrt{k}}$-Bernoulli subset
is  $3\kappa$-dominating and has size at most $\frac{2}{\sqrt{k}} |V|$ with  probability at least $1-\kappa$.
\end{lemma}
\begin{proof}
Suppose  \eqref{eq:k_M}  holds.
Let $\mathcal{G}=(V,E)$ be a graph satisfying the assumptions in the statement of the lemma,
 and let $W \subset V$ be $\frac{1}{\sqrt{k}}$-Bernoulli.

For $v \in V$, let $n(v)$ be number of edges $(v,w) \in E$ with  $w \in W$. The random variable $n(v)$ is Binomial $B(\frac{1}{\sqrt{k}},\deg_{\mathit{out}}(v))$.
Let
$$Y := \sum_{v \in V} 1_{[n(v)=0]}.$$
It follows that
$$E(Y) = \sum_{v \in V} P\left( n(v) = 0\right) = \sum_{v \in V~,~ \deg_{\mathit{out}}(v) < k } P\left( n(v) = 0\right) + \sum_{v \in V~,~ \deg_{\mathit{out}}(v) \ge  k } P\left( n(v) = 0\right).$$
Thus
$$E(Y)  \le \kappa|V|  +\left(1-\frac{1}{\sqrt{k}}\right)^k|V|< 2\kappa |V|.$$
%Thus the expected number of vertices that are not adjacent to  some $w \in W$  is at most $$\left(\epsilon +(1-\frac{1}{\sqrt{k}})^k\right)|V|< 2\epsilon |V|.$$
For $v,w \in V$, the random variables $n(v)$ and $n(w)$ are independent, unless there is a common vertex $u \in V$ which both $(v,u) \in E$ and $(w,u) \in E$. Because the maximal in-degree is at most $M$, each $u \in V$ can account for at most $M^2$ such pairs, so  there are at most $M^2|V|$ pairs which are not independent.
Also note that $\mathit{Var}(1_{[n(v)=0]}) \le 1$ for every $v \in V$ so $\mathit{Cov}(1_{[n(v)=0]},1_{[n(w)=0]})\le 1$.
It follows that
$$\mathit{Var}(Y) = \sum_{v,w \in V}\mathit{Cov}(1_{[n(v)=0]},1_{[n(w)=0]})\le M^2 |V|.$$

%Thus the variance of the number of vertices with no edge directed at some $w \in W$ is at most
%$M^2|V|$.

By  Chebyshev's inequality,
the probability that $W$ is not  $3\kappa$-dominating is at most
$$P\left( Y > 3\kappa |V|) \le P( |Y - E(Y)| > \kappa |V|\right) \le \frac{\mathit{Var}(Y)}{\kappa^2|V|^2} \le \frac{M^2}{\kappa^2|V|} < \frac{\kappa}{2}.$$
Also $E(|W|) = \frac{1}{\sqrt{k}}|V|$ and
$\mathit{Var}(|W|) < \frac{1}{\sqrt{k}}|V|$, so again by Chebyshev's inequality
$$P( |W| > \frac{2}{\sqrt{k}} |V|) \le \frac{k\mathit{Var}(|W|)}{ |V|^2} =\frac{1}{ \sqrt{k} |V|} \le \frac{\kappa }{2}.$$
It follows that with probability at least $1- \kappa$, $W$ is  $3\kappa$-dominating and $|W| \le \frac{2}{\sqrt{k}} |V|$.
%the number of vertices covered is concentrated around the expectation.
\end{proof}

\section{Proof of Theorem \ref{thm:positive_entropy_almost_free}}

% Suppose the action $G \curvearrowright (X,\mu)$ is ergodic.
%Because  the function $x \mapsto |\stab(x)|$ is $G$-invariant, by ergodicity  it is constant off a $\mu$-null set, and in particular  it is either finite with probability $1$ or infinite with probability $1$.
Suppose $\stab(x)$ is infinite $\mu$-almost-surely. %, and fix $\epsilon >0$.
  Our goal is to prove that the sofic entropy of the $G$-action is non-positive with respect to any sofic approximation (in the case of a deterministic approximation sequence this means it is either $0$ or $-\infty$).
   By a direct inspection of the definition of sofic entropy in \eqref{eq:sofic_entropy_def}, our goal is to show that for  any $\eta >0$ given there exists a neighborhood $\OO \subset \Prob(X)$ of $\mu$ so that for any sufficiently good approximation $\xi:G \to S_V$,
   $$\frac{1}{|V|}\log \sep_{\eta}(\Map(\OO,\xi),d^V) < \eta.$$

We will show that we can choose the neighborhood $\OO \subset \Prob(X)$ to be of the form
$\OO= \OOO$ (see Definition \ref{def:OO} below), for some  parameters $F_1,F_2 \Subset G$ and $\epsilon, \delta >0$.

\begin{defn}\label{def:approx_stab}\textbf{(Approximate stabilizer)}
For $F \Subset G$ and $\delta >0$ and $x \in \chi^G$ let
\begin{equation}\label{eq:stab_delta_F_def}
\stab_{\delta,F}(x) := \bigcap_{h \in F}\{ g \in G ~:~ d(x_h,g(x)_h) < \delta \}.
\end{equation}
\end{defn}

%$$\stab_{\delta}(x) = \{ g \in G ~:~ d(x_1,g(x)_1) < \delta \}$$
%Fix $M>0$, $\epsilon >0$, $F_0,F_1,F_2 , F \subset G$.

% $M>0$, $\epsilon >0$, $F_0,F_1,F_2 , F \subset G$.

\begin{defn}\label{def:OO}
Let  $\epsilon, \delta , M >0$ and $F_1,F_2\Subset G$. Define
$$\OO[M,\delta,\epsilon,F_1,F_2] \subset \Prob(X)$$  to be the set of probability measures $\nu \in \Prob(X)$ satisfying the following conditions:

\begin{equation}\label{eq:stab_big}
\nu(\{ x~:~ |\stab_{\delta,F_1}(x) \cap F_1| < M \}) < \epsilon
\end{equation}
%\begin{equation}\label{eq:almost_stab_stab}
%\nu(\{ x~:~ (\stab_{\delta,F_1}(x) \cap F_1) \ne (\stab_{\delta,F_2}(x) \cap F_0) \}) < \epsilon
%\end{equation}
\begin{equation}\label{eq:almost_stab_stab}
\nu\left(\left\{ x~:~ (\stab_{\delta,F_1}(x) \cap F_1) \ne (\stab_{\delta,F_2^2}(x) \cap F_1) \right\}\right) < \epsilon
\end{equation}
%\item $\nu\left(\right\{ x~:~ (\stab_{\delta,F_1}(g(x)) \cap F_0) \ne ((\stab_{\delta,F_1}(x))^g \cap F_0)\left\}\right) < \epsilon$ (*** this follows from $(2)$***)

\end{defn}

\begin{lemma}\label{lem:OOO_is open}
If  $\delta^{-1}$ is not an integer power of $2$, the set
 $\OOO \subset \Prob(X)$ is open.
%{\huge \textbf{Check!}}
\end{lemma}
\begin{proof}
Suppose $\delta^{-1}$ is not an integer power of $2$. By \eqref{eq:d_range} it follows that  $d(\omega,\omega') < \delta$ if and only if $d(\omega,\omega')  \le \delta$.
So for every $F \Subset G$,
$$\stab_{\delta,F}(x) = \bigcap_{h \in F}\{ g \in G ~:~ d(x_h,g(x)_h) <  \delta\} =  \bigcap_{h \in F}\{ g \in G ~:~ d(x_h,g(x)_h) \le \delta\}.$$
It follows that for any $g \in G$ and $F \Subset G$ the set
$\{x \in X~:~ g \in \stab_{\delta,F}(x)\}$ is a \textbf{clopen} set: It is both open and closed in $X$.

Because $F_1$ and $F_2$ are both finite,
$$A :=  \left\{ x \in X~:~  (\stab_{\delta,F_1}(x) \cap F_1) \ne (\stab_{\delta,F_2^2}(x) \cap F_1) \right\}$$
%{\huge \textbf{ Not Open!}}
 and
$$B := \left\{ x \in X~:~ |\stab_{\delta,F_1}(x) \cap F_1| < M \right\}$$ are  also clopen in $X$.
So the indicator functions $1_A,1_B: X  \to \mathbb{R}$ are continuous.
Now
$$\OOO = \left\{ \mu \in \Prob(X)~:~ \int 1_A(x) d\mu(x) < \epsilon \mbox{ and } \int 1_B(x) d\mu(x) < \epsilon \right\},$$
%The indicator function of a clopen set is continuous.
% $\OOO$ is the set of measures in $\Prob(X)$ that assign measure less than $\epsilon$ to each of these clopen sets in $X$,
 so $\OOO \subset \Prob(X)$ is an open set.
% {\huge \textbf{In general if $U \subset X$ is open then $U^*_{< \epsilon}:=\{ \mu \in \Prob(X)~:~ \mu(U) < \epsilon\}$ needs not be open. For example take $X=[0,1]$, $U= (0,1)$ then %$\delta_1 \in \partial U^*_{< \epsilon} \cap U^*_{ < \epsilon}$.}}
\end{proof}

We now specify how to choose the parameters  $\epsilon >0$, $\delta >0$, $M >0$ $F_1 \Subset G$ and $F_2 \subset G$ are chosen according to $\eta$:

\begin{itemize}
%\item Start with an arbitrary $\eta >0$.
\item Choose $\epsilon$ so that
\begin{equation}\label{eq:cond_epsilon}
0 < \epsilon < \min\{\frac{\eta}{100},\frac{1}{3}\}.
\end{equation}
\item Choose $\delta >0$, so that $\delta^{-1}$ is not an integer power of $2$ and a finite subset $F_0 \Subset G$ depending on $\epsilon$ and on the measure $\mu$ so that
\begin{equation}\label{eq:cond_delta} %\label{eq:mu_stab_delta_F_1}
\mu \left(\left\{x \in X~:~ \stab_{\delta,F_0}(x) \ne \stab(x)  \right\} \right) < \epsilon/2,
\end{equation}
Where  $\stab_{\delta,F_0}(x)$ is defined  in Definition \ref{def:approx_stab} above.
The is possible by Lemma \ref{lem:cond_delta_holds} below.
We also require
\begin{equation}\label{eq:cond_delta_2}
3\delta < \eta - 100\epsilon
\end{equation}
\item Choose $M >0$ depending on $\epsilon$  and $\delta$ big enough so that
\begin{equation}\label{eq:cond_M}
\sup_{n >M}(1- \frac{1}{\sqrt{n}})^n < \epsilon/3 \mbox{ and }
\frac{4}{\sqrt{M}}\log \sep_{\delta/2}(\chi,d) < \frac{\eta}{2}.
\end{equation}
It is clear that the left hand side in both expressions tends to $0$ as $M \to \infty$, so such choice of $M$  is indeed possible.
\item Choose a finite subset $F_1 \Subset G$ depending on $M$ on $\epsilon$ and on the measure $\mu$ so that $F_0 \cup \{1\} \subset  F_1$ and
\begin{equation}\label{eq:cond_F1}
\mu\left(\left\{  x\in X~:~ |\stab(x) \cap F_1| \le M \right\}\right) < \epsilon/2.
\end{equation}
We prove the existence of such a set $F_1$  in Lemma \ref{lem:cond_F_1_holds} below.
\item Choose another finite subset $F_2  \Subset G$ so that $F_1 \subset F_2$, $F_2 = F_2^{-1}$ and
\begin{equation}\label{eq:cond_F2}
\frac{2}{\sqrt{|F_2|}}|F_1|\log(2) < \frac{\eta}{2}
\end{equation}
\item Choose $V$ big enough so that
\begin{equation}\label{eq:cond_V}
|V| > 2 |F_2|^2(\epsilon/3)^{-3}.
\end{equation}
 % \item $ 1_G \in  F_1 \subset F_2 = F_2^{-1} \subset F_2^2$.
%  \item $(1- \frac{1}{\sqrt{|F_2|}})^M < \epsilon$.
\item Choose  $\xi:G \to S_V$ to be a $(F_2^6,\epsilon/3)$-approximation of $G$.
\end{itemize}

\begin{lemma}\label{lem:cond_delta_holds}
For any measure $\mu \in \Prob(X)$ and $\epsilon >0$ \eqref{eq:cond_delta} holds for some $F_0 \subset G$ and sufficiently small $\delta >0$.
\end{lemma}
\begin{proof}
Note that
\begin{equation}\label{eq:stab_approx_stab_intersect}
\stab(x) = \bigcap_{\delta >0}\bigcap_{F \Subset G} \stab_{\delta,F}(x).
\end{equation}
Also, $\stab_{\delta_1,F_1}(x) \subset \stab_{\delta_2,F_2}(x)$ whenever $\delta_1 \le  \delta_2$ and $F_2 \subseteq F_1$.
So by $\sigma$-additivity of $\mu$,
$$ \inf_{\delta > 0~,~F \Subset G}\mu\left(\left\{x \in X~:~ \stab_{\delta,F}(x) \ne \stab(x)  \right\} \right) =  0.$$
It follows that \eqref{eq:cond_delta} holds for some $F_0 \subset G$ and sufficiently small $\delta >0$.
\end{proof}

\begin{lemma}\label{lem:cond_F_1_holds}
Under the assumption that $\stab(x)$ is infinite $\mu$-almost-surely, for every $M >0$ and $\epsilon >0$ there exists $F_1 \Subset G$ so that $F_0 \cup \{1\} \subset F$ and \eqref{eq:cond_F1} holds.
\end{lemma}
\begin{proof}
Because $\stab(x) = \infty$ $\mu$-a.e, it follows that for any $M >0$,

$$ \mu\left(\left\{x \in X~:~ |\stab(x)| \le M  \right\} \right) =0,$$
Note that
$$ \left\{x \in X~:~ |\stab(x)| \le  M  \right\}  =\bigcap_{F \Subset G} \left\{x \in X~:~ |\stab(x) \cap F| \le M  \right\},$$
So as in the proof of Lemma \ref{lem:cond_delta_holds} using  $\sigma$-additivity of $\mu$, it follows that
\eqref{eq:cond_F1} holds for some $F_1 \Subset G$. Furthermore, we can assume that $F_0 \cup \{1\} \subset F_1$ by further increasing $F_1$.
\end{proof}

\begin{lemma}\label{lem:mu_in_OOO}
For $M >0$, $\epsilon,\delta >0$ and $F_1,F_2 \Subset G$ as above,
$\mu \in \OOO$.	
\end{lemma}

\begin{proof}
Because $F_1 \subset F_2$, it follows that
$$ \stab_{\delta,F_1}(x) \subseteq \stab_{\delta,F_2^2}(x) \subseteq \stab(x).$$
So by \eqref{eq:cond_delta} it follows that \eqref{eq:almost_stab_stab} also holds with $\nu$ replaced by $\mu$.
Using \eqref{eq:cond_delta} and \eqref{eq:cond_F1}  and the condition $F_0 \subset F_1$ we see that \eqref{eq:stab_big} holds with $\nu$ replaced by $\mu$.

Thus  $\mu \in \OOO$.

\end{proof}

%The general idea of the proof is as follows: Suppose $x \in \chi^V$ is a good approximation of the action $G \curvearrowright (X,\mu)$. %Then for most $v \in V$, there are many $g \in G$ that are ``approximate stabilizers'' in the sense that $x_v$ and $x_{\xi_g(v)}$ are %very close. Once we have a ``good guess'' of the ``approximate stabilizer'' for many such $v$'s, we can ``approximately recover'' %$x_v$ for most $v \in V$ by sampling $x$ on a relatively small set $D \subset V$.
%One should take care, however, that in order to be able to choose a set $D$ that is a small fraction of $V$, the ``approximate stabilizers'' %need to be big, so a priori there might be many possibilities for them.
%To guess the ``approximate stabilizer'' for most $v \in D$, we use the fact that the ``approximate stabilizer'' of $x_v$ is %``approximately'' obtained from the ``approximate stabilizer'' of $x_{\xi_g(v)}$ by conjugating it with $g$. Thus,
% it is possible to use conjugation to interpolate  ``approximate stabilizers'' on $D \subset V$ by sampling  on an even smaller set $C %\subset V$.

%We now make this sketch precise.

%We first define an ``approximate stabilizer'':

The following lemma shows that approximate stabilizers behave well under conjugation:
\begin{lemma}\label{lem:almost_stab_stab_implies_stab_conj}
If $F_1 \subset F_2=F_2^{-1}$ and $x\in X$ satisfies
\begin{equation}\label{eq:almost_stab_stab_pointwise}
(\stab_{\delta,F_1}(x) \cap F_1) = (\stab_{\delta,F_2^2}(x) \cap F_1)
\end{equation}
then
%\begin{equation}\label{eq:stab_conj}
%\forall g \in F_2  ~(\stab_{\delta,F_1}(g(x)) \cap F_0) = g^{-1}(\stab_{\delta,F_1}(x) )g \cap F_0
%\end{equation}
\end{lemma}
\begin{equation}\label{eq:stab_conj}
\forall g \in F_2  ~g(\stab_{\delta,F_1}(x) \cap F_1)g^{-1} \subseteq  \stab_{\delta,F_1}(g(x))
\end{equation}

\begin{proof}
Suppose \eqref{eq:almost_stab_stab_pointwise} holds.

Choose any
$f \in \stab_{\delta,F_1}(x) \cap F_1$. By  \eqref{eq:almost_stab_stab_pointwise},
\begin{equation}\label{eq:stab_F_2}
d(x_{h}, x_{f^{-1}h})<\delta  ~ \forall h \in F_2^2.
\end{equation}
Now choose any $g \in F_2$.
For any $h \in F_1$  we have $g^{-1}h \in F_2^{-1}F_1 \subset F_2^2$ so we can substitute $g^{-1}h$ instead of $h$ in \eqref{eq:stab_F_2} to obtain
$$d(x_{g^{-1}h},x_{f^{-1}g^{-1}h})<\delta.$$
Now $(g(x))_h = x_{g^{-1}h}$ and
$$(gfg^{-1}g(x))_h = x_{f^{-1}g^{-1}h}.$$
So we have
$$ d((g(x))_h,(gfg^{-1}g(x))_h) < \delta.$$

This means that  $(gfg^{-1}) \in \stab_{\delta,F_1}(g(x))$.

We conclude that  \eqref{eq:almost_stab_stab_pointwise} implies \eqref{eq:stab_conj}.
\end{proof}

\begin{defn}\label{defn:good_v}
%For parameters $M >0$, $\delta >0 $,  $F_1,F_2 \Subset G$,
Call $v \in V$ \emph{good for $x \in\chi^V$} if the following conditions are satisfied:
\begin{equation}\label{eq:v_good_for_x_1}
|\stab_{\delta,F_1}(\Pi_v^{\xi}(x)) \cap F_1| \ge  M
\end{equation}
\begin{equation}\label{eq:v_good_for_x_2}
\stab_{\delta,F_1}(\Pi_v^{\xi}(x)) \cap F_1 = \stab_{\delta,F_2^2}(\Pi_v^{\xi}(x))  \cap F_1
\end{equation}
and
\begin{equation}\label{eq:v_good_for_x_3}
\xi_{g_1}\circ \xi_{g_2}\circ \xi_{g_3}(v) = \xi_{g_1g_2g_3}(v)  ~\forall g_1,g_2,g_3 \in F_2^4.
\end{equation}
Otherwise, say that $v \in V$ is \emph{ bad for $x \in \chi^V$}.
\end{defn}

%From now on, we assume:

\begin{lemma}\label{lem:good_tau_exists}
%For any  $\epsilon>0$, $\delta>0$ and $F_1 \Subset G$ there exists $F_2 \Subset G$  with the following property:
%Suppose $\xi:G \to S_V$ is a sufficiently good sofic approximation for $G$ and that $|V|$ is big enough depending on the other %parameters.
Let  $\Omega \subset \Map(\OOO,\xi)$ with $2 \le |\Omega| < \infty$.
Then there exists a set $C \subset V$ and a function $\tau:V \to F_2$ with the following properties:
\begin{enumerate}
\item[(I)] $|C| < \frac{2}{\sqrt{|F_2|}}|V|$.
\item[(II)] $\frac{1}{|V|} |\{v \in V ~:~ \xi_{\tau}(v) \not \in C\}| < 2\epsilon$
\item[(III)] %$\frac{1}{|\Omega|}\left| \left\{ x\in \Omega~:~  \frac{1}{|V|} \sum_{v \in V}\delta_{\Pi_x^\xi(\xi_{\tau}(v))} \in \OO[M,\delta,\epsilon,F_0,F_1,F_2] \right\} \right|> \frac{1}{2} $,
$\frac{1}{|\Omega|}\left| \left\{ x\in \Omega~:~  |\{ v\in V~:~ \xi_\tau(v) \mbox{  is bad for } x\}| < 8\epsilon|V|   \right\}\right| \ge  \frac{1}{2}$.
%For every $x \in \Map(\OOO,\xi)$ the following holds: $$|\{ v\in V~:~ \xi_\tau(v) \mbox{  is bad for } x\}| < \kappa|V|$$
\end{enumerate}
where $\xi_{\tau}:V \to V$  is defined by
\begin{equation}
\xi_{\tau}(v):=\xi_{\tau(v)}(v).
\end{equation}

\end{lemma}
\begin{proof}
%Choose $F_2 \Subset G$ so that $F_2= F_2^{-1}$ and so that $|F_2|$ is big enough as a function of $\epsilon$.  % to satisfy some %requirements that will appear below.
%Now let  $\xi:G \to S_V$ be a sufficiently hood sofic approximation, depending on $F_2$, $\epsilon$ and the other parameters.
%We will also assume that
%The various requirements from  $|F_2|$,  $|V|$ and $\xi:G \to S_V$ will be stated when we use them.
%***
Consider the directed graph $\mathcal{G}_{\xi,F_2}=(V,E)$ with
$$E = \left\{ (u,v) \in V \times V~:~ \exists g \in F_2 \mbox{ s.t. } \xi_g(u)=v\right\}.$$

Because the approximation $\xi:G \to S_V$ is symmetric, the maximal in-degree in $\mathcal{G}_{\xi,F_2}$ is at most $|F_2|$.
Let
$V'  \subset V$ denote the set of $v$'s for which the mapping $g \mapsto \xi_g(v)$  is injective on $F_2$.
%Assume $\xi:G \to S_V$ is a $(F_2^2,\frac{\kappa}{6}|F_2|^{-2})$-approximation. Keeping in mind that $F_2= F_2^{-1}$,  it follows that
%the for a pair $h,g \in F_2$ with $h \ne g$:
%$$\frac{1}{|V|} \left|\left\{ v \in V~:~ \xi_g(v) = \xi_h(v)  \right\}\right| =
%\frac{1}{|V|} \left|\left\{ v \in V~:~ \xi_g(\xi_{h^{-1}}(v)) = \xi_h(\xi_{h^{-1}}(v))  \right\}\right| =$$
%$$=\frac{1}{|V|} \left|\left\{ v \in V~:~ \xi_g(\xi_{h^{-1}}(v)) = v  \right\}\right| \le$$
%$$\le  \frac{1}{|V|} \left|\left\{ v \in V~:~ \xi_g(\xi_{h^{-1}}(v))  \ne  \xi_{gh^{-1}}(v)  \right\}\right| +$$
%$$+  \frac{1}{|V|} \left|\left\{ v \in V~:~  \xi_{gh^{-1}}(v) = v  \right\}\right|  \le  \frac{\kappa}{3} |F_2|^{-2}$$
%$$ + \frac{1}{|V|} \#\left\{ v \in V~:~  \xi_{gh^{-1}(v) = v  \right\} \le 2 \kappa |F_2|^{-2}.$$
%Taking a union bound over unordered pairs $h,g \in F_2$ we get
Because  $\xi:G \to S_V$ is a sufficiently good approximation of $G$ it follows that
$|V \setminus V'| \le \frac{\epsilon}{3} |V|$, so  $\mathcal{G}_{\xi,F_2}$  is $(\epsilon/3,|F_2|,|F_2|)$-regular.

By Lemma \ref{lem:small_random_set_dominating}, a $\frac{1}{\sqrt{|F_2|}}$-Bernoulli set $C \subset V$ is  $\epsilon$-dominating set $C \subset V$ and has size less than $\frac{2}{\sqrt{|F_2|}} |V|$ with probability at least $1-\frac{2}{3}\epsilon < \frac{1}{2}$.
To see that Lemma  \ref{lem:small_random_set_dominating} applies, we used the left inequality in \eqref{eq:cond_M} (keeping in mind that $|V| > M$), and \eqref{eq:cond_V} to  deduce that \eqref{eq:k_M} is satisfied with $k = M = |F_2|$ and $\kappa$ replaced with $\epsilon/3$ and $N=|V|$.
In this case   $C\subset V$ satisfies $(I)$.
For $v \in V$ choose $\tau:V \to F_2$ randomly as follows: Whenever the set $N_v := \{g \in  F_2~:~ \xi_g(v) \in C \}$  is non-empty,
choose $\tau(v)$ uniformly at random from $N_v \subset F_2$. If $N_v = \emptyset$ let $\tau(V)$ be chosen uniformly at random from $F_2$.
We see that if  $C$ is $\epsilon$-dominating $(II)$ is satisfied.

To conclude the proof we will show that $(III)$ is satisfied with probability at least $1/2$.

For $x \in \Omega$ and $v \in V$ denote:
\begin{equation}\label{eq:good_set}
\Psi_{x,v} := \begin{cases}
0&
 \mbox{ if } v \mbox{  is good for } x  \\
1 & \mbox{ if } v  \mbox{ is bad for } x
\end{cases}
\end{equation}

Because of $\Omega \subset \OOO$,  it follows that for any $x \in \Omega$, all but an $\epsilon$-fraction of the $v$'s are good so
\begin{equation}\label{eq:stab_wrong_all}
 \frac{1}{|V|}\sum_{v \in V} \Psi_{x,v} < \epsilon ~ \forall x \in \Omega.
\end{equation}
%So we have:
%$$\int_{C,\tau} \sum_{\phi \in \Map} E_{[d]}\Psi_{\phi,g}\circ \tau  d\nu(C,\tau) \le |\Map| 4 \delta |F_0|.$$

Now let $Z_{x,v}$ denote the indicator of the event ``$\xi_\tau(v)$ is bad for $x$''
$$Z_{x,v}:= \sum_{g \in F_2} 1_{[\tau(v)= g]}\Psi_{x,\xi_g(v)}.$$

$Z_{x,v}$ is a random variable, because $\tau:V \to F_2$ is a random function.

Note that
\begin{equation}\label{eq:prob_F_2}
\PP\left(\tau(v)=g\right)=|F_2|^{-1}~ \forall v \in V' \, , g \in F_2.
\end{equation}
It follows that  for $v \in V'$, $$\EE (Z_{x,v}) = \frac{1}{|F_2|} \sum_{g \in F_2} \Psi_{x,\xi_g(v)}.$$
%So
%$$
%\EE\left(\frac{1}{|V|} \sum_{v \in V} Z_{x,v} \right) = \EE\left(\frac{1}{|V|} \sum_{v \in V} Z_{x,v} \right) +
%\EE\left(\frac{1}{|V|} \sum_{v \in V} Z_{x,v} \right)$$
Because $|V \setminus V'| < \epsilon |V|$ it follows that
\begin{equation}\label{eq:exptation_xi_tau}
\EE\left(\frac{1}{|V|} \sum_{v \in V} Z_{x,v}\right) \le \frac{1}{|V|}\sum_{v \in V} \frac{1}{|F_2|}\sum_{g \in F_2} \Psi_{x,\xi_g(v)}  + \epsilon.
\end{equation}
Because $\xi_g \in S_V$ is a permutation:
$$\sum_{v \in V}  \Psi_{x,\xi_g(v)} = \sum_{v \in V}  \Psi_{x,v}.$$  So from \eqref{eq:exptation_xi_tau} and \eqref{eq:stab_wrong_all} we get that for every $x \in \Omega$
$$
\EE\left(\frac{1}{|V|} \sum_{v \in V} Z_{x,v}\right) \le \frac{1}{|V|}\sum_{v \in V}\Psi_{x,v}   + \epsilon \le 2\epsilon.
$$

Averaging over $x \in \Omega$:
$$\EE \left(\frac{1}{|\Omega|}\sum_{x \in \Omega} \Av Z_{x,v}\right)  =
\frac{1}{|\Omega|}\sum_{x \in \Omega}\EE  \left[\frac{1}{|V|}\sum_{v \in  V} {\tau} Z_{x,v}\right] \le 2\epsilon$$

Using Markov inequality, it follows that

$$\PP \left[  \left(\frac{1}{|\Omega|}\sum_{x \in \Omega} \Av Z_{x,v} \right) > 4\epsilon \right] \le \frac{1}{2}.$$
So $(III)$ holds with probability at least $\frac{1}{2}$.
\end{proof}

Given  a metric space $(Y,\rho)$ and a finite set $V$, the following ``hamming-like'' metric  is defined on $Y^V$:
\begin{equation}
\rho^V(x,y) := \Av d(x_v,y_v).
\end{equation}

We also have the following ``uniform'' metric $d_\infty^D$ on $\chi^D$, where $D$ is  a  finite set:
\begin{equation}\label{d_infty}
d_\infty^D(x,y) := \max_{v \in D} d(x_v,y_v).
\end{equation}
%We will use the  following simple lemma:
We will use the following relatively estimate:
%. It should be regarded as some kind of subadditivity for $V \mapsto \log \sep_\delta (X^V,d^V)$ :
\begin{lemma}\label{lem:cov_products}
%Suppose $\diam(X)=1$. Then
%$$\log \cov_{2\delta}(X^V,d^V) \le \left(H\left(\frac{\delta}{2+\delta}\right) + \log\cov_{\delta}(X,d) \right)|V|$$
For any finite set $D$and $\delta >0$ we have
$$ \log \sep_{2\delta}(\chi^D,d^D_\infty) \le |D| \log \sep_\delta(\chi,d)$$
\end{lemma}
\begin{proof}
If $S \subset \chi$ is such that  $\chi =  \bigcup_{x \in S} B_\delta(x)$ and $|S|= \cov_\delta(\chi,d)$
then the union of $\delta$-balls in $(\chi^D,d^D_\infty)$ with centers in $S^V$ covers $\chi^V$.
It follows that
$$\log \cov_{\delta}(\chi^D,d^D_\infty) \le |D| \log\cov_{\delta}(\chi,d).$$
The claim now follows by
the following standard and easily verified facts:
$$  \sep_{2\delta}(\chi^D,d^D_\infty) \le \cov_{\delta}(\chi^D,d^D_\infty)
\mbox{ and }
\cov_{\delta}(\chi,d) \le \sep_{\delta}(\chi,d)$$

\end{proof}

We record the following Lemma (see   \cite[Lemma $3.1$]{austin2015additivity}, and recall that we use a left-action):
\begin{lemma}\label{lem:names_almost_equivariant}
Suppose $v \in V$ is good for $x \in \chi^V$ and $g \in F_2^3$ then
$$g^{-1} (\Pi_{\xi_{g}(v)}^\xi(x)) \mid_{F_2^3} =  \Pi_v^\xi(x) \mid_{F_2^3}.$$
\end{lemma}
\begin{proof}
Because $v$ is good for $x$ it follows that
$$ \xi_{h^{-1}g^{-1}}(\xi_{g}(v)) = \xi_{h^{-1}}(v) ~ \forall h \in F_2^3,$$
so for every $h \in F_2^3$ we have
$$g^{-1} (\Pi_{\xi_{g}(v)}^\xi(x))_h =  (\Pi_{\xi_{g}(v)}^\xi(x))_{g h} = x_{\xi_{h^{-1}g^{-1}}(\xi_{g}(v))} = x_{\xi_{h^{-1}}(v)} = (\Pi_v^\xi(x))_h.$$
\end{proof}

The following lemma is the heart of our proof of Theorem \ref{thm:positive_entropy_almost_free}:
\begin{lemma}\label{lem:finitary_stablizer_kill_entropy}
%Fix $\eta>0$, $0 < \kappa < \eta/9$, and $F_1 \Subset G$. For every  $0 < \delta < \frac{1}{3}(\eta-9\kappa)$ and sufficiently big $M %>1$ there exists $F_2 \Subset G$ and  $\epsilon >0$ so that for every
%sufficiently good sofic approximation $\xi:G \to S_V$

The following holds:

\begin{equation}\label{eq:sep_OO_bound}
\frac{1}{|V|} \log \sep_{\eta}(\Map(\OOO,\xi),d^V) < \frac{4}{\sqrt{|F_2}} \cdot |F_1| \log(2) +  \frac{4}{\sqrt{M}} \log \sep_{\delta/2} (\chi,d)
\end{equation}
\end{lemma}
\begin{proof}

%We will assume that $F_2$ is big enough, that  $\epsilon >0$ is small enough to satisfy in particular
%\begin{equation}\label{eq:eta_big}
%\eta > 3\delta +   \left[2(|F_2|^3+2) \epsilon+ 9
%\kappa\right],
%\end{equation}
% and that $\xi:G \to S_V$ is a good enough approximation.
%In particular we assume the parameters satisfy the assumptions of Lemma \ref{lem:good_tau_exists}. Other requirements from the %parameters will be stated when we use them.

Fix any  subset  $\Omega \subset \Map(\OOO,\xi)$ that is $\eta$-separated with respect to the metric $d^V$.
%Let $\Omega$ be as in \eqref{eq:def_Omega}.
Let $\tau:V \to F_2$ and $C \subset V$  be as in the conclusion of Lemma \ref{lem:good_tau_exists}.
By condition $(III)$ there exists $\Omega' \subset \Omega$
%so that $|\Omega'| \ge \frac{1}{2}|\Omega|$ and
\begin{equation}\label{eq:Omega_prime_good}
\frac{1}{|V|}|\{ v\in V~:~ \xi_\tau(v) \mbox{  is bad for } x\}| < 8 \epsilon ~ \forall x \in \Omega'.
\end{equation}
 Denote  $S := (2^{F_1})^C$, and for each  $s \in S$, let
 $$\Omega_s := \left\{ x \in \Omega'~:~ (\stab_{\delta,F_1}(\Pi_v^\xi(x)) \cap F_1)= s_v \forall v \in C\right\}.$$

Then
$\Omega' = \bigcup_{s \in S}\Omega_s'$, so
\begin{equation}
  |\Omega| \le 2\cdot  2^{|F_1|\cdot |C|} \cdot \max_{s \in S}|\Omega_s|.
\end{equation}

By $(I)$,  $|C| \le \frac{2}{\sqrt{|F_2|}} |V|$. It follows that
\begin{equation}\label{eq:sep_UB1}
  \frac{1}{|V|}\log |\Omega|\le \frac{4}{\sqrt{|F_2|}}   |F_1| \log(2) + \max_{s \in S}\frac{1}{|V|}\log |\Omega_s|.
\end{equation}

So our next goal is to bound $ |\Omega_s|$, for $s \in S$.  %Fix $s \in S$.

For $s \in S$ and $v \in V $ define:
\begin{equation}\label{eq:stab_v_s}
\stab(v,s) := \begin{cases}
(\tau(v))^{-1} s_{\xi_\tau(v)} \tau(v) & \xi_\tau(v) \in C\\
\emptyset & \mbox{otherwise}
\end{cases}
\end{equation}
We claim that if $x \in \Omega_s$ and $v$, $\xi_\tau(v)$ are both good for $x$  then
\begin{equation}\label{eq:stab_delta_subset_stab_v_s}
\stab(v,s) \subset \stab_{\delta,F_1}(\Pi_v^\xi(x)).
\end{equation}

Indeed, we can assume  $\xi_\tau(v) \in C$ otherwise $\stab(v,s) = \emptyset$ and \eqref{eq:stab_delta_subset_stab_v_s} holds trivially.
 Then
$$s_{\xi_\tau(v)} = \stab_{\delta,F_1}(\Pi_{\xi_\tau(v)}^\xi(x)) \cap F_1.$$
Denote $g_v := \tau(v)$. Because $v$ is good for $x$ and $g_v \in F_2$, by Lemma \ref{lem:names_almost_equivariant},
$$g_v^{-1} (\Pi_{\xi_{g_v}(v)}^\xi(x)) \mid_{F_2^3} =  \Pi_v^\xi(x) \mid_{F_2^3}.$$
So

$$\stab_{\delta,F_1}\left( g_v^{-1} (\Pi_{\xi_{g_v}(v)}^\xi(x))\right) \cap F_2^3=
\stab_{\delta,F_1}\left(\Pi_v^\xi(x)\right) \cap F_2^3.$$
Because $\xi_\tau(v)$ is good for $x$    \eqref{eq:v_good_for_x_2} holds with $v$ replaced by $\xi_\tau(v)$. So by Lemma \ref{lem:almost_stab_stab_implies_stab_conj} applied with $g= \tau(v)^{-1}$,
$$\tau(v)^{-1} s_{\xi_\tau(v)} \tau(v) \subset \stab_{\delta,F_1}(g_v^{-1}(\Pi_{\xi_\tau(v)}^\xi(x))) \cap F_2^3= \stab_{\delta,F_1}(\Pi_v^\xi(x)) \cap F_2^3.$$
This proves \eqref{eq:stab_delta_subset_stab_v_s} holds.

Consider the  graph $\mathcal{G}_{s}'=(V,E_s)$
where
$$ E_s := \{ (v,g(v)) ~:~ g \in \stab(v,s)\}.$$

%whose vertex set is $V$ and whose  edges are the pairs $(v,w)$  for which:
%\begin{itemize}
%\item $\xi_\tau(v) \in C$.
%\item There exists  some $g \in \stab(v,s)$ so that $\xi_g(v)=w$ ,  $\xi_{g \tau(v)}(v) %= \xi_{g}(\xi_\tau (v))$ and $\xi_{\tau(v)g\tau(v)^{-1}}(v) = \xi_{\tau(v)} (\xi_g %(\xi_\tau^{-1}(v)))$.
%\end{itemize}
%Now let $\mathcal{G}_s$ be the symmetric graph obtained from $\mathcal{G}_{s}'$. %That is $(v,w)$ is an edge in $\mathcal{G}_s$ if and only if both $(v,w)$ and $(w,v)$ %are edges in $\mathcal{G}_s'$.

\noindent\textbf{Claim A:}  If $(v,w)$ is an edge in $\mathcal{G}_s$ and $x \in \Omega_s$ and $v,\xi_\tau(v)$ are both good for $x$ then $d(x_v,x_w) < \delta$.

\noindent\textbf{Proof of Claim A:}
By definition of  $\mathcal{G}_s$  there exists   $g \in \stab(v,s)$   so that $\xi_g(v) = w$.
By the argument above $g \in \stab_{\delta,F_1}(\Pi_v^\xi(x))$, so $d( (\Pi_v^\xi(x))_1,g(\Pi_v^\xi(x))_1) < \delta$.
Now $x_v =  (\Pi_v^\xi(x))_1$ and $$x_w = x_{\xi_g(v)} = (\Pi_v^\xi(x))_{g^{-1}} = g((\Pi_v^\xi(x))_1,$$
so indeed $d(x_v,x_w) < \delta$.
%By Lemma \ref{lem:almost_stab_stab_implies_stab_conj} it follows that
%$\tau(v) g \tau(v)^{-1} \in \stab_{\delta,F_1}(\Pi_x^\xi(v))$.  Because $v$ is good for $x$, \eqref{eq:v_good_for_x_2} and \eqref{eq:v_good_for_x_3} together imply that $d(x_v,x_w) < \delta$. This proves Claim A.

\noindent\textbf{Claim B:} The graph $\mathcal{G}_s$ is $(11\epsilon,M,|F_2|^3)$-regular.

%All but $14\kappa |V|$ of the vertices in $\mathcal{G}_s$ have degree at least $M$.

\noindent\textbf{Proof of Claim B:}

Note that by definition $\stab(v,s) \subset F_2^{-1} F_1 F_2 \subset F_2^3$ , so $(u,v) \in E_s$ implies that $v = \xi_g(u)$ for some $g \in F_2^3$. This shows that $\mathcal{G}_s$ has maximal in-degree at most $|F_2|^3$.

The properties of $C$, $\tau$ and $\Omega'$ assure that
$$|\{v \in V~:~ \xi_\tau(v) \not \in C\}| < 2 \epsilon|V|$$
 and
$$\forall x \in \Omega' ~ |\{v \in V~:~ \xi_\tau(v) \mbox{ is bad for }  x\}| <8 \epsilon  |V|.$$

It follows that $|\stab(v,x) | < M$ on at most $10 \epsilon |V|$ $v$'s. Also, as in the proof of Lemma \ref{lem:good_tau_exists},
because  $\xi:G \to S_V$ is a sufficiently good sofic approximation the map $g \mapsto \xi_g(v)$ is injective on $F_2^3$ for all but at most $\epsilon |V|$ $v$'s.  It follows that  at most
$11 \epsilon|V|$ of the vertices in $\mathcal{G}_s$ have degree smaller than $M$.
This completes the proof of  Claim B.

By \eqref{eq:cond_M} and \eqref{eq:cond_V}, the condition \eqref{eq:k_M} is satisfied with $M$ replaced by $|F_2|^3$, $k$ replaced by $M$ and $\kappa$ replaced by $11 \epsilon$. So using Claim B
we can apply   Lemma \ref{lem:small_random_set_dominating} to deduce that there is a set $D \subset V$ of size at most $\frac{2}{\sqrt{M}}|V|$ which is $33\epsilon$-dominating in $\mathcal{G}_s$.
As in the proof of Lemma \ref{lem:good_tau_exists}, there exists a function
$\tau':V \to F_2^3$ so that for all but $33\epsilon$ $v$'s
$(v,\xi_{\tau'(v)}(v))$ is an edge in $\mathcal{G}_s$ and $\xi_{\tau'(v)}(v) \in D$.
%Also, as in the proof of Lemma \ref{lem:good_tau_exists},
%$$\forall x \in \Omega_s~ \frac{1}{|V|}|\left\{ v \in V~:~ \xi_{\tau'(v)}(v) \mbox{ is bad for } x \right\}| \le  |F_2|^3 \epsilon.$$

\noindent\textbf{Claim D:} If $x,y \in \Omega_s$ and $d(x_v,y_v) < \delta$ for all $v \in D$ then $d^V(x,y) < \eta$.

\noindent\textbf{Proof of Claim D:}
Suppose $x,y \in \Omega_s$ and $d(x_v,y_v) < \delta$ for all $v \in D$.
Fix $w \in V$. Denote $v = \xi_{\tau'(v)}(v)$.  If $w$ and $\xi_{\tau}(w)$ are both good for both $x$ and for $y$ and $(w,v)$ is an edge in $\mathcal{G}_s$, it follows from Claim A that $d(x_w,x_v) < \delta$ and $d(y_w,y_v) < \delta$.
Furthermore, if $v \in D$, then $d(x_v,y_v) < \delta$ so  in that case $d(x_w,y_w) < 3\delta$.
It follows that $d(x_w,y_w)  > 3\delta$ implies that either $\xi_{\tau'(v)}(v) \not \in D$ or one of $w,\xi_{\tau}(w)$ is not good for $x$ or for $y$. Thus
$$|\left\{ w \in V~:~ d(x_w,y_w) > 3\delta \right\}| < 2(33+ 8) \epsilon|V| < 100\epsilon |V| .$$
Thus, because the diameter of $\chi$ is bounded by $1$,
$$\Av d(x_v,y_v) \le 3\delta +   100\epsilon< \eta,$$
where in the last inequality we used \eqref{eq:cond_delta_2}.
This completes the proof of Claim D.

%Now recall that $\eta > 3\delta +   \left[2(|F_2|^3+2) \epsilon+ 9
%\kappa\right]$.

%Let $\pi_D:\chi^V \to \chi^D$ denote the obvious projection given by the restriction map.
Because $\Omega_s$ is  $\eta$-separated, Claim D implies that the restriction map  $\pi_D:\chi^V \to \chi^D$
is injective on $\Omega_s$, and that
$\pi_D(\Omega_s)$ is $\delta$-separated with respect to the metric $d^D_\infty$.
Thus by Lemma \ref{lem:cov_products},
$$\log |\Omega_s| = \log | \pi_D(\Omega_s)| \le \log \sep_{\delta}(\chi^d,d^D_\infty) \le |D| \log \sep_{\delta/2} (\chi,d).$$

We conclude that
\begin{equation}\label{eq:Omega_s_prime_bound}
\log |\Omega_s| \le  \frac{4}{\sqrt{M}}|V| \sep_{\delta/2} (\chi,d),
\end{equation}

Together with \eqref{eq:sep_UB1} this shows that
$$\frac{1}{|V|}\log |\Omega| \le  \frac{4}{\sqrt{|F_2|}}\cdot |F_1| \log(2) +  \frac{4}{\sqrt{M}} \log \sep_{\delta/2} (\chi,d).$$
Since $\Omega$ was an arbitrary $\eta$-separated subset of $\OOO$, this completes the proof.

\end{proof}

To conclude the proof of
Theorem \ref{thm:positive_entropy_almost_free}, observe that the right hand side of \eqref{eq:sep_OO_bound} is bounded above by $\eta$ because of \eqref{eq:cond_F2} and the right inequality in \eqref{eq:cond_M}.
\section{Finite stabilizers and completely positive entropy}
We conclude with  a corollary regarding actions with completely positive $\Sigma$-entropy, due to  Weiss  \cite{MR2052281} in the amenable case.
Recall that  $G \curvearrowright (Y,\nu)$ is a \textbf{factor} of $G \curvearrowright (X,\mu)$, that is there is a $G$-equivariant map $\pi:X \to Y$ with $\nu = \mu \circ \pi^{-1}$.
An action $G \curvearrowright (X,\mu)$ of a sofic group has \textbf{completely positive $\Sigma$-entropy} if any non-trivial factor has positive $\Sigma$-entropy.

Recall that the an action $G \curvearrowright (X,\mu)$ is \textbf{faithful} if $\mu(\{ x\in X~:~ g(x) \ne x\})>0$ for all $g \in G$.

\begin{cor}\label{cor:completely_positive_entropy}
Let $G$ be an infinite countable sofic group. If an ergodic action $G \curvearrowright (X,\mu)$ is faithful and has completely positive  entropy with respect to some sofic approximation $\Sigma$, it is free.
\end{cor}

\begin{proof}
By Theorem \ref{thm:positive_entropy_almost_free}, the stabilizers must be finite, thus  by Lemma \ref{prop:finite_stab} the map $x \mapsto \stab(x)$ induces a finite factor.  But an action of an infinite group on finite probability space  must have infinite stabilizers. In particular by Theorem \ref{thm:positive_entropy_almost_free} this factor has zero entropy. Because $G \curvearrowright (X,\mu)$ has completely positive sofic entropy it follows that $\stab(x)$ is constant, and because the action is faithful it must be trivial, so the action is free.
\end{proof}

\bibliographystyle{abbrv}
\bibliography{sofic}
\end{document}